\documentclass{amsart}
\usepackage{amssymb,graphicx}
\usepackage{xcolor}

\theoremstyle{plain}
\newtheorem{theorem}{Theorem}[section]
\newtheorem{proposition}[theorem]{Proposition}
\newtheorem{lemma}[theorem]{Lemma}
\newtheorem{cor}[theorem]{Corollary}

\theoremstyle{remark}
\newtheorem{rem}[theorem]{Remark}
\newtheorem{ex}[theorem]{Example}
\newtheorem{definition}[theorem]{Definition}

\newcommand{\F}{{\mathcal F}}
\newcommand{\N}{{\mathbb N}}
\newcommand{\Z}{{\mathbb Z}}
\newcommand{\Q}{{\mathbb Q}}
\newcommand{\C}{{\mathbb C}}
\newcommand{\R}{{\mathbb R}}
\newcommand{\K}{{\mathbb K}}

\numberwithin{equation}{section}

\newcommand{\av}{{\bf a}}

\newcommand{\dv}{{\bf d}}

\newcommand{\vv}{{\bf v}}

\newcommand{\nullv}{{\bf 0}}

\newcommand{\DD}{\mathcal{D}}

\newcommand{\FF}{\mathcal{F}}
\newcommand{\GG}{\mathcal{G}}
\newcommand{\HH}{\mathcal{H}}

\newcommand{\NN}{\mathcal{N}}
\newcommand{\OO}{\mathcal{O}}

\renewcommand{\SS}{\mathcal{S}}

\newcommand{\Cc}{\mathbb{C}}
\newcommand{\Rr}{\mathbb{R}}
\newcommand{\Qq}{\mathbb{Q}}

\newcommand{\Zz}{\mathbb{Z}}

\newcommand{\Kk}{\mathbb{K}}

\newcommand{\medfrac}[2]{\mbox{\large{$\textstyle{\frac{#1}{#2}}$}}}
\newcommand{\medbinom}[2]{\mbox{\large{$\textstyle{\binom{#1}{#2}}$}}}

\newcommand{\half}{\mbox{$\textstyle{\frac{1}{2}}$}}

\renewcommand{\ge}{\geq}
\renewcommand{\le}{\leq}
\newcommand{\kdots}{,\ldots ,}

\newcommand{\ie}{\emph{i.e.},\ }

\renewcommand{\mod}[3]{#1\equiv#2\,({\rm mod}\,#3)}

\renewcommand{\gcd}{{\rm gcd}}

\def\house#1{\setbox1=\hbox{$\,#1\,$}%
\dimen1=\ht1 \advance\dimen1 by 2pt \dimen2=\dp1 \advance\dimen2
by 2pt
\setbox1=\hbox{\vrule height\dimen1 depth\dimen2\box1\vrule}%
\setbox1=\vbox{\hrule\box1}%
\advance\dimen1 by .4pt \ht1=\dimen1 \advance\dimen2 by .4pt
\dp1=\dimen2 \box1\relax}

\renewcommand{\theta}{\vartheta}

\newcommand{\oz}{\!\otimes_{\mathbb{Z}}\!}

\newcommand{\rank}{{\rm rank}\,}

\begin{document}

\title[Number systems over general orders]{Number systems over general orders}
\author[J.-H. Evertse]{Jan-Hendrik~Evertse}
\author[K. Gy\H{o}ry]{K\'alm\'an~Gy\H{o}ry}
\author[A. Peth\H{o}]{Attila~Peth\H{o}}
\author[J.M. Thuswaldner]{J\"org~M.~Thuswaldner}
\address{J.-H.~Evertse, Leiden University, Mathematical Institute, P.O. Box 9512, 2300 RA Leiden, THE NETHERLANDS}
\email{evertse@math.leidenuniv.nl}
\address{K.~Gy\H{o}ry, Institute of Mathematics, University of Debrecen,
H-4002 Debrecen, P.O. Box 400, HUNGARY}
\email{gyory@science.unideb.hu}
\address{A.~Peth\H{o}, Department of Computer Science, University of Debrecen,
H-4010 Debrecen, P.O. Box 12, HUNGARY and\\
University of Ostrava, Faculty of Science,
Dvo\v{r}\'{a}kova 7, 70103 Ostrava, CZECH REPUBLIC}
\email{Petho.Attila@inf.unideb.hu}
\address{J.M.~Thuswaldner, Chair of Mathematics and Statistics,
University of Leoben, Franz-Josef-Strasse 18, A-8700 Leoben,
AUSTRIA}
\email{Joerg.Thuswaldner@unileoben.ac.at}
\thanks{Research of K.Gy. was supported in part by the OTKA grants NK115479. Research of A.P. was supported in part by the grant no. 17-02804S of the Czech Grant Agency. Research of J.T. is supported by the FWF grants P27050 and P29910.}
\date{\today}

\begin{abstract}
Let $\OO$ be an order, that is a commutative ring with $1$
whose additive structure is a free $\Zz$-module of finite rank.
A \emph{generalized number system} (GNS for short) over $\OO$ is a pair
$(p,\DD )$ where $p\in\OO [x]$ is monic with constant term $p(0)$
not a zero divisor of $\OO$, and where $\DD$ is a complete
residue system modulo $p(0)$ in $\OO$ containing $0$.
We say that $(p,\DD )$ is a
GNS over $\OO$ with the finiteness property if all elements of
$\OO [x]/(p)$ have a representative in $\DD [x]$ (the polynomials
with coefficients in $\DD$). Our purpose is to extend several
of the results from a previous paper of Peth\H{o} and
Thuswaldner,
where GNS over orders of number fields were considered.
We prove that it is algorithmically decidable whether or not for
a given order $\OO$ and GNS $(p,\DD )$ over $\OO$,
the pair $(p,\DD )$ admits the finiteness property.
This is closely related to work of Vince
on matrix number systems.

Let $\FF$ be a fundamental
domain for $\OO\oz\Rr/\OO$ and $p\in \OO [X]$ a monic polynomial.
For $\alpha\in\OO$, define $p_{\alpha}(x):=p(x+\alpha )$ and
$\DD_{\FF ,p(\alpha )}:= p(\alpha )\FF\cap\OO$.
Under mild conditions we show that the pairs
$(p_{\alpha},\DD_{\FF,p(\alpha)})$ are GNS over $\OO$ with finiteness
property provided $\alpha\in\OO$ in some sense approximates
a sufficiently large positive rational integer.
In the opposite direction we prove under different conditions that
$(p_{-m},\DD_{\FF ,p(-m)})$ does not have the finiteness property for each large enough positive rational integer $m$.

We obtain important relations between power integral bases
of \'{e}tale orders, and GNS over $\Zz$.
Their proofs depend on some general effective finiteness results of Evertse and Gy\H{o}ry on monogenic \'{e}tale orders.
\end{abstract}

\date{\today}

\maketitle

\section{Introduction}\label{s1}
Decimal and sexagesimal representations of the positive integers have been used since the times of antiquity. A computer's ``native language'' consists of the binary strings,
which can be interpreted among others as binary representations of integers. Starting with the pioneering work of V. Gr\"unwald \cite{Grunwald:1885} many generalizations have been established. For an overview we refer to the papers \cite{ABBPT_I,PT:2017} and to the book~\cite{Evertse_Gyory}.

In the present paper, $\OO$ will denote an order, that is a commutative ring
with $1$ whose additive group is free abelian of finite rank.
We identify $m\in\Zz$ with $m\cdot 1$, and thus assume $\Zz\subset\OO$.
The order $\OO$ may be given explicitly by a basis $\{ 1=\omega_1,\omega_2\kdots\omega_d\}$ and a multiplication table
\begin{equation}\label{eq:multiplicationtable}
\omega_i\omega_j =\sum_{l=1}^d a_{ijl}\omega_l\ \
(i,j=2\kdots d)\ \ \text{with } a_{ijl}\in\Zz,
\end{equation}
satisfying the commutativity and associativity rules.
A \emph{generalized number system} over $\OO$ (GNS over $\OO$ for short)
is a pair $(p,\DD )$, where $p\in\OO [x]$ is a monic polynomial
such that $p(0)$ is not a zero divisor of $\OO$,
and where $\DD$ is a (necessarily finite) complete residue system of $\OO$ modulo $p(0)$
containing $0$.

An element $a\in \mathcal{O}[x]$ is \emph{representable in $(p,\mathcal{D})$} if there exist an integer $L\ge 0$ and $a_0,\dots,a_L\in \mathcal{D}$ such that
\begin{equation} \label{eq:kongruencia}
a \equiv \sum_{j=0}^L a_j x^j \pmod p,
\end{equation}
\ie if there is $q\in\OO [x]$ such that $a+qp$
has its coefficients in $\DD$.
Our condition that $p(0)$ not be a zero divisor of $\OO$ implies
that a representation of $a\pmod p$ as in \eqref{eq:kongruencia},
if it exists, is unique (except for ``leading zeros'').
If all $a\in \mathcal{O}[x]$ are representable in $(p,\DD)$, then $(p,\DD)$ is called a GNS with the \emph{finiteness property}.
This concept was introduced for 
$\OO =\Zz$ by Peth\H{o} \cite{pethoe1991:polynomial_transformation_and} and extended to orders $\OO$ in number fields by Peth\H{o} and Thuswaldner \cite{PT:2017}.

GNS over orders may be viewed as
special cases of \emph{matrix number systems}, which
were introduced by Vince \cite{Vince}. A matrix number system
is a triple $(\Lambda ,\varphi ,D)$, consisting of a
free abelian group $\Lambda$ of finite rank, an injective homomorphism
$\varphi :\Lambda\to\Lambda$, and a complete residue system
$D$ for $\Lambda/\varphi (\Lambda )$ with $0\in D$.
Then a GNS $(p,\DD )$ over $\OO$ may be viewed as a matrix
number system with $\Lambda =\OO [x]/(p)$,
$\varphi: f\pmod{p}\mapsto x\cdot f\pmod{p}$, and $D=\DD$.

We briefly recall some of the results from the paper \cite{PT:2017}
of Peth\H{o} and Thuswaldner, but
reformulate them in terms of the language of the present paper.
Let $\OO$ be an order of a number field.
We embed $\OO$ in the $\Rr$-algebra $\OO\oz\Rr$
and view $\OO$ as a full rank sublattice of $\OO\oz\Rr$.
The real innovation of \cite{PT:2017} is to consider
parametrized classes of GNS over $\OO$
determined by a monic polynomial $p\in\OO [x]$ and a
fundamental domain $\FF$ of $\OO\oz\Rr/\OO$
with $0\in\FF$, \ie $\FF$ is a subset of $\OO\oz\Rr$
consisting of precisely one element from every residue class
of $\OO\oz\Rr$ modulo
$\OO$. More precisely, one considers GNS of the type
$(p_{\alpha},\DD_{\FF ,p(\alpha )})$ ($\alpha\in\OO$),
where $p_{\alpha}(x):=p(x+\alpha )$ and
$\DD_{\FF ,\theta}:= \theta \FF\cap\OO$ for a non-zero element
$\theta$ of $\OO$.
In their paper, Peth\H{o} and Thuswaldner
proved that it is algorithmically decidable whether
a given GNS $(p,\DD)$ over $\OO$ has the finiteness property. Under mild conditions on $\FF$ they were able to prove that $(p_{\alpha}, \DD_{\FF,p(\alpha)})$ is a GNS with the finiteness property provided
that there is a large positive rational integer $m$ such that
$m^{-1}\alpha$ is close to $1$ (with respect to any vector norm
on $\OO\oz\Rr$).
Under different conditions on $\FF$ they proved that $(p_{-m},\DD_{\F,p(-m)})$ does not have the finiteness property
if $m$ is a sufficiently large positive rational integer.
These are far reaching generalizations of results of
Kov\'acs and Peth\H{o}~\cite{KovacsPetho}.
The purpose of the present paper is to extend the results
mentioned above from GNS over orders of number fields
to GNS over arbitrary orders. In particular, that for a given GNS
the finiteness property is effectively decidable is an easy
consequence of general work of Vince \cite{Vince}
on matrix number systems.



\vskip0.2cm\noindent
{\bf Outline of the paper.}
Let $\OO$ be an arbitrary order.
In Section \ref{s2} we show that the finiteness property of a given
GNS $(p,\DD )$ over $\OO$ is effectively decidable by applying some of Vince's
results on matrix number systems \cite{Vince}.
In Section \ref{s3} we define the digit sets $\DD_{\FF ,\theta}$
using fundamental domains $\FF$ of $\OO\oz\Rr/\OO$ and prove a
sufficient condition for a GNS over $\OO$ to admit the finiteness
property ({\it cf.} Theorem \ref{th:1}).
In Section \ref{s4} we prove that if $p\in \OO[x]$ and the fundamental domain $\F$ satisfies some mild condition then the pairs $(p_{\alpha},\DD_{\FF,p(\alpha )})$ are always GNS with
the finiteness property provided
there is a large positive rational integer
$m$ such that $m^{-1}\alpha$ is close to $1$ (this is the content of Theorem~\ref{th:newKovacs}).

Section \ref{s5}
is devoted to GNS without the finiteness property. The main result is Theorem~\ref{non-ECNSmain} which states that $(p_{-m}, \DD_{\FF,p(-m)})$ does not have the finiteness property for all large enough positive rational integers $m$.

Using some general effective finiteness results of Evertse and Gy\H{o}ry \cite[Corollary8.4.7]{Evertse_Gyory} on monogenic orders in \'etale algebras ({\it cf.} also Proposition~\ref{E-Gy} in Section \ref{s6} below), we obtain important relations between power integral bases and number systems in \'{e}tale orders (see Theorem~\ref{thm:newKovacsPetho}), which in turn can be interpreted
as GNS over $\Zz$.


\section{Connection with matrix number systems}\label{s2}

Recall that a matrix number system is a triple $(\Lambda ,\varphi ,D)$,
consisting of a free abelian group $\Lambda$ of finite rank,
an injective $\Zz$-linear homomorphism $\varphi:\, \Lambda \to\Lambda$,
and a complete residue system $D$ for
$\Lambda/\varphi (\Lambda)$, containing $\nullv$.
The rank of this matrix number system is the
rank of $\Lambda$, and its determinant is
the cardinality of $\Lambda/\varphi (\Lambda )$
(which is equal to the absolute value of the determinant of $\varphi$).
Denote by $\chi_{\varphi}$ the characteristic polynomial
of $\varphi$, \ie
$\det (x\cdot{\rm id}-\varphi )$.
This characteristic polynomial is monic of degree equal to the rank
of $\Lambda$, with coefficients in $\Zz$ and non-zero
constant term $\pm\det\varphi$.

We say that $(\Lambda,\varphi ,D)$ has the finiteness property
if every $\vv\in\Lambda$ can be expressed as a finite sum
$\sum_{i=0}^L \varphi^i\dv_i$, with $\dv_i\in D$ for $i=0\kdots L$.
Such systems were introduced by Vince \cite{Vince} (he used
a different terminology). We recall some of Vince's results.

Let $(\Lambda ,\varphi ,D)$ be a matrix number system.
Here, and similarly in other situations below,
we identify elements of $\Lambda$
with their images in the real vector space $\Lambda\oz\Rr$ under the canonical embedding.
Thus, we view $\Lambda$ as a lattice (discrete subgroup of maximal
rank) in $\Lambda\oz\Rr$,
and if $\{ \av_1\kdots\av_d\}$ is a $\Zz$-basis of $\Lambda$,
it is also an $\Rr$-vector space basis of $\Lambda\oz\Rr$.
We endow $\Lambda\oz\Rr$ with a vector norm $\|\cdot\|$;
this induces a norm on $\Lambda$.

\begin{proposition}\label{vince-1}
Assume that $(\Lambda ,\varphi ,D)$ has the finiteness property.
Then $\chi_{\varphi}$ is expansive, \ie all its
complex roots have absolute value $>1$.
\end{proposition}

\begin{proof}
See Vince \cite[p. 508, Prop. 4]{Vince}.
\end{proof}

Assume henceforth that $\chi_{\varphi}$ is expansive.
For $\vv\in\Lambda$,
we define the sequences $(\vv_i)_{i=0}^{\infty}$
in $\Lambda$
and $(\dv_i)_{i=0}^{\infty}$ in $D$ inductively by
\begin{eqnarray*}
&&\vv_0:=\vv ;
\\
&&\dv_i\in D\ \text{is the representative of }
\vv_i ({\rm mod}\ \varphi (\Lambda )),
\\
&&\vv_{i+1}:=\varphi^{-1}(\vv_i-\dv_i)\ \ \
(i=0,1,2,\ldots ).
\end{eqnarray*}
In an appropriate completion of $\Lambda$ we can now write
$\vv =\sum_{i=0}^{\infty} \varphi^i\dv_i$
and call this the $(\varphi ,D)$-expansion of $\vv$.
Vince \cite[p.~511 AlgorithmA and p.~512, Lemma~2]{Vince}
observes that there is an effectively
computable number $C>0$ depending on $\Lambda, \varphi, D, \vv$
such that $\|\vv_i\|\leq C$ for all $i$. This implies that the
sequences $(\vv_i)_{i=0}^{\infty}$
and $(\dv_i)_{i=0}^{\infty}$ are ultimately periodic, where
for both sequences the sum of the lengths of the preperiod and period
is bounded above by the number $R$ of points in $\Lambda$
of norm at most $C$.
Further, it is clear that $\vv$ has a finite $(\varphi, D)$-expansion
$\vv =\sum_{i=0}^L \varphi^i\dv_i$ if and only if $\vv_i=\dv_i=\nullv$
for $i>L$. So we may take $L<R$. This shows that for given
$\vv\in\Lambda$, we can effectively decide whether it has a finite
$(\varphi ,D)$-expansion, and if such an expansion exists,
it has length $\leq R$.

The following result implies that it can be effectively decided
whether $(\Lambda ,\varphi ,D)$ has the finiteness property.

\begin{proposition}\label{vince-2}
Let $(\Lambda ,\varphi ,D)$ be a matrix number system.
There is an effectively computable number $C'>0$ depending
on $\Lambda$, $\varphi ,D$ such that the following are equivalent:
\\
(i) $(\Lambda ,\varphi ,D)$ has the finiteness property;
\\
(ii) $\chi_{\varphi}$ is expansive, and every $\vv\in\Lambda$ with $\|\vv\|\leq C'$ has a finite $(\varphi ,D)$-expansion.
\end{proposition}

\begin{proof}
See Vince \cite[p.~513, Theorem~4]{Vince}.
\end{proof}

We now specialize the above to generalized number systems.
Let $\OO$ be an order, and $(p,\DD )$ a GNS over $\OO$.

Let $\DD [x]$ denote the set of polynomials with
coefficients in $\DD$,
and $R(p,\DD )$ the set of $a\in\OO [x]$
such that $\mod{a}{b}{p}$ for some $b\in\DD [x]$.

With the usual identification
of an element of $\OO$ with its canonical image in
the  $\Rr$-algebra $\OO\oz\Rr$,
we view $\OO$ as a full rank
lattice in $\OO\oz\Rr$.

We endow the $\Rr$-algebra $\OO\oz\Rr$ with a vector norm $\|\cdot\|$.
For instance, fixing a $\Zz$-module basis $\{ 1=\omega_1,\omega_2\kdots\omega_d\}$ of $\OO$,
we can express every $\alpha\in\OO\oz\Rr$ as $\sum_{i=1}^d x_i\omega_i$
for some $x_1\kdots x_d\in\Rr$ and we may take $\|\alpha\|:=\max_i|x_i|$.
The elements of $\OO$ are those with $x_1\kdots x_d\in\Zz$, hence
for given $C$, the set of $\alpha\in\OO$ with $\|\alpha\|\leq C$
is finite and effectively determinable. But in what follows the choice of
norm doesn't matter. We define the norm $\| a\|$ of $a\in\OO [x]$ to be the
maximum of the norms of its coefficients.

As already mentioned in Section \ref{s1}, we can view the GNS $(p,\DD )$ as
a matrix number system $(\Lambda ,\varphi ,D)$, where
$\Lambda =\OO [x]/(p)$,
$\varphi : f\pmod{p}\mapsto x\cdot f\pmod{p}$ and $D=\DD$.
To see this, observe that $\Lambda/\varphi(\Lambda )\cong \OO [x]/(p,x)\cong \OO/p(0)\OO$, so that indeed $\DD$ is a complete residue system for $\Lambda/\varphi(\Lambda )$. Further,
one easily verifies that a congruence $\mod{a}{\sum_{i=0}^L d_ix^i}{p}$
translates into $a\pmod{p}=\sum_{i=0}^L\varphi^id_i$.
Using what we observed above for matrix number systems,
this shows that for given $a\in\OO [x]$ it can be decided effectively
whether it has a finite expansion $\sum_{i=0}^L d_ix^i\pmod{p}$.
Further, $(p,\DD )$ has the finiteness property if and only if
$(\Lambda ,\varphi ,\DD )$ has the finiteness property.

We may view $\OO[x]$ as a free $\Zz[x]$-module of finite rank,
and $a\mapsto p\cdot a$ as a $\Zz [x]$-linear map from
$\OO[x]$ to itself. The determinant of this $\Zz [x]$-linear map is a monic polynomial in $\Zz [x]$, which we denote by $Np$.


\begin{lemma}\label{lem:charpol}
$Np =\chi_{\varphi}$.
\end{lemma}

\begin{proof}
Let $d=\rank\OO$, $n=\deg p$.
Pick a $\Zz$-basis $1=\omega_1,\omega_2\kdots\omega_d$ of $\OO$.
Let $p=x^n+p_{n-1}x^{n-1}+\cdots +p_0$.
For $i=0\kdots n-1$, let $P_i$ be the matrix of the $\Zz$-linear map
$\alpha\mapsto p_i\alpha$ from $\OO$ to $\OO$ with respect to the
basis of $\OO$ just chosen.
Clearly, $\{ \omega_ix^j:\, i=1\kdots d,\, j=0\kdots n-1\}$
is a $\Zz$-basis of $\Lambda =\OO [x]/(p)$. A straightforward
computation shows that the matrix of $\varphi$ with respect
to this basis is
\[
\begin{pmatrix}
0&\cdots&0&-P_0
\\
I&&0&-P_1
\\
&\ddots&&\vdots
\\
0&&I&-P_{n-1}
\end{pmatrix},
\]
where each entry represents a $d\times d$ integer matrix and $I$ denotes the identity matrix of order $d$. Hence
\begin{eqnarray*}
\chi_{\varphi}&=&
\left|
\begin{matrix}
xI&&&P_0
\\
-I&xI&&P_1
\\
&\ddots&&\vdots
\\
&&-I&xI+P_{n-1}
\end{matrix}\right|
=
\left|
\begin{matrix}
\ 0&\cdots&\ 0&x^nI+x^{n-1}P_{n-1}+\cdots +P_0
\\
-I&&\ 0&*
\\
&\ddots&&\vdots
\\
\ 0&&-I&*
\end{matrix}\right|
\\[0.2cm]
&=&\det(x^nI+x^{n-1}P_{n-1}+\cdots+xP_1 +P_0)=Np.
\end{eqnarray*}
\end{proof}

Using division with remainder, it follows that every
element of $\Lambda =\OO [x]/(p)$ can be represented by a polynomial
in $\OO [x]$ of degree $<n$.
Using this and Lemma \ref{lem:charpol},
the following effective finiteness criterion for a GNS having the finiteness
property is a straightforward translation of Proposition \ref{vince-2}.

\begin{theorem} \label{non-ECNS1}
Let $(p,\DD )$ be a GNS over $\OO$ with $\deg p=n\geq 1$.
Then there is an effectively computable
number $C''$, depending on $\OO$, $p$ and $\DD$, such that the following are
equivalent:
\\
(i) $(p,\DD )$ has the finiteness property;
\\
(ii) the polynomial $Np$ is expansive, and every $a\in\OO [x]$
with $\| a\|\leq C''$, $\deg a <n$ belongs to $R(p,\DD )$.
\end{theorem}

If $\OO$ is an order in a number field $K$,
$p=x^n+p_{n-1}x^{n-1}+\cdots +p_0\in\OO [x]$, and $\alpha\mapsto\alpha^{(i)}$ ($i=1\kdots d)$ are the embeddings of $K$ in $\Cc$,
then $Np=\prod_{i=1}^d (x^n+p_{n-1}^{(i)}x^{n-1}+\cdots +p_0^{(i)})$. Thus Theorem \ref{non-ECNS1} is a generalization of Theorem 2.9 of Peth\H{o} and Thuswaldner \cite{PT:2017}.

\section{Digit sets defined by fundamental domains} \label{s3}

Let $\OO$ be an order of rank $d$.
We view $\OO$ as a full rank sublattice of the $\Rr$-algebra
$\OO\oz\Rr$.
We frequently use the simple fact that an element
$\theta\in\OO$ is not a zero divisor of $\OO$
if and only if it is invertible in $\OO\oz\Rr$ (consider
the map $\alpha\mapsto\theta \alpha$).

Recall that a fundamental domain for $\OO\oz\Rr/\OO$
is a subset of $\OO\oz\Rr$ containing precisely one element
from every residue class of $\OO\oz\Rr$ modulo $\OO$.
For a fundamental domain $\FF$ for $\OO\oz\Rr/\OO$ with $0\in\FF$
and
$\theta\in\OO$ which is not a zero divisor, we define
\[
\DD_{\FF ,\theta} :=\theta\FF\cap\OO =\{ \alpha\in\OO :\, \theta^{-1}\alpha\in\FF\}.
\]

The following two lemmas are easy generalizations of
\cite[Lemmas 2.3, 2.4]{PT:2017}.

\begin{lemma}\label{lem:fund-1}
Let $\FF$ be a fundamental domain for $\OO\oz\Rr/\OO$ with $0\in\FF$
and $\theta\in\OO$ not a zero divisor. Then $\DD_{\FF ,\theta}$
is a complete residue system for $\OO$ modulo $\theta$ containing $0$.
\end{lemma}

\begin{proof} Let $\alpha\in\OO$. Then $\theta^{-1}\alpha =\xi +\beta$
with $\xi\in\FF$ and $\beta\in\OO$. Thus,
$\alpha =\theta\xi +\theta\beta$ with $\theta\xi\in\DD_{\FF ,\theta}$.
If $\delta_1,\delta_2$ are two elements of $\DD_{\FF ,\theta}$
with $\delta_1-\delta_2\in\theta\OO$, then $\theta^{-1}\delta_1,\theta^{-1}\delta_2$ are elements of $\FF$ whose difference lies in $\OO$, so $\delta_1=\delta_2$. This proves Lemma \ref{lem:fund-1}.
\end{proof}

\begin{lemma}\label{lem:fund-2}
Let $\theta\in\OO$ not be a zero divisor and let $\DD$ be a complete residue system for $\OO$ modulo $\theta$ with $0\in\DD$.
Then there is a bounded fundamental domain $\FF$ of $\OO\oz\Rr/\OO$
such that $\DD =\DD_{\FF ,\theta}$.
\end{lemma}

\begin{proof} Let $1=\omega_1,\omega_2,\ldots,\omega_d$ be a $\Z$-basis of $\OO$, and set $\GG :=\{u_1\theta^{-1}\omega_1+\cdots+ u_d\theta^{-1}\omega_d\;:\; 0\le u_j<1, j=1,\ldots d\}$. Then $\GG$ is a bounded fundamental domain for
$\OO\oz\Rr/\theta^{-1}\OO$ with $0\in\GG$. The set
\[
\FF :=\bigcup_{\delta\in\DD} (\theta^{-1}\delta +\GG )
\]
is obviously bounded. We see that it satisfies the other requirements of the lemma. First observe that
\[
\DD_{\FF ,\theta}=\theta\FF\cap\OO =
\bigcup_{\delta\in\DD} (\delta +\theta\GG )\cap\OO =\DD ,
\]
for $(\delta +\theta\GG )\cap\OO =\{ \delta\}$ for every $\delta\in\OO$
since $\theta\GG$ is a fundamental domain for $\OO\oz\Rr/\OO$
containing $0$.

It remains to show that $\FF$ is a fundamental domain
for $\OO\oz\Rr/\OO$.
We can express any $\eta\in\OO\oz\Rr$ as $\eta =\xi +\theta^{-1}\alpha$
with $\xi\in\GG$, $\alpha\in\OO$, and we have $\alpha =\delta +\theta\beta$
with $\delta\in\DD$, $\beta\in\OO$, thus,
\[
\eta = (\theta^{-1}\delta +\xi )+\beta =\zeta +\beta\ \
\mbox{with } \zeta\in\FF,\ \beta\in\OO.
\]
Further, if $\zeta_1,\zeta_2$ are elements of $\FF$ with
$\zeta_1-\zeta_2\in\OO$,
then writing $\zeta_i=\theta^{-1}\delta_i+\xi_i$ with $\delta_i\in\DD$,
$\xi_i\in\GG$ for $i=1,2$, we see that
\[
\zeta_1-\zeta_2=\theta^{-1}(\delta_1-\delta_2)+(\xi_1-\xi_2).
\]
Thus, $\xi_1-\xi_2\in\theta^{-1}\OO$, hence $\xi_1=\xi_2$, hence
$\theta^{-1}(\delta_1-\delta_2)\in\OO$, hence $\delta_1=\delta_2$,
and finally, $\zeta_1=\zeta_2$. So indeed $\FF$ is a bounded fundamental domain for $\OO\oz\Rr/\OO$.
\end{proof}

We choose any vector norm on $\OO\oz\Rr$ and endow $\OO\oz\Rr$
with the corresponding topology; this topology does not
depend on the chosen norm. Given a subset $\SS$ of $\OO\oz\Rr$,
we denote by $\overline{\SS}$ the closure of $\SS$ with respect
to this topology.

Let $\FF$ be a fundamental domain for $\OO\oz\Rr/\OO$
with $0\in\FF$ and such that $\FF$ is bounded.
We call two elements $\alpha ,\beta$ of $\OO$
$\FF$-\emph{neighbours} of one another if $(\alpha +\overline{\FF})\cap (\beta +\overline{\FF})\not=\emptyset$.
Let $\Delta_{\FF}\subset\OO$ be the set of all $\FF$-neighbours of $0$;
this set contains in particular $0$ itself. Further, if $\alpha\in\Delta_{\FF}$,
then so is $-\alpha$.

\begin{lemma}\label{lem:neighborbasis}
The set of $\FF$-neighbors $\Delta_{\FF}$ of $0$ contains a $\Zz$-basis of $\OO$.
\end{lemma}

\begin{proof}
See \cite[Lemma 2.6]{PT:2017} for the case that $\OO$ is an order
in a number field. The proof given there works in the same way
for arbitrary orders, replacing $\Zz^k$, $\NN$ by $\OO$, $\Delta_{\FF}$.
\end{proof}

Let $p=x^n+p_{n-1}x^{n-1}+\cdots +p_0\in\OO [x]$ such that $p_0$
is not a zero divisor of $\OO$, and put $p_n:=1$. We define the set
\begin{equation}\label{eq:deltaz1}
Z_{\FF ,p} :=\bigg\{
\sum_{j=1}^n \delta_j p_j \;:\; \delta_j\in \Delta_{\FF}
\bigg\}.
\end{equation}

\begin{theorem}\label{th:1}
Let $(p,\DD )$, with $p=x^n+p_{n-1}x^{n-1}+\cdots +p_0\in\OO [x]$ be a GNS over $\OO$ and let $\FF$ be a bounded fundamental domain for $\OO\oz\Rr/\OO$ with $\DD_{\FF ,p_0}=\DD$.
Assume that the following conditions hold (setting $p_n:=1$):
\begin{itemize}
\item[(i)] $Z_{\FF ,p} + \DD \subset \bigcup_{\delta \in \Delta_{\FF}} (\DD + p_0\delta)$, \vskip 2mm
\item[(ii)] $Z_{\FF ,p} \subset \DD \cup (\DD - p_0)$, \vskip 2mm
\item[(iii)] $\left\{ \sum_{j\in J} p_j \;:\; J\subseteq\{1,\ldots, n\}  \right\} \subseteq \DD$.
\end{itemize}
Then $(p,\DD )$ is a GNS over $\OO$ with the finiteness property.
\end{theorem}

\begin{proof}
See \cite[Theorem~3.1]{PT:2017} for a proof in the case that $\OO$
is an order in a number field. The proof carries over without
modifications to arbitrary orders.
\end{proof}

\begin{rem}
For $p\in \Z[x]$ denote by $L(p)$ the sum of the absolute values of the coefficients of $p$. Akiyama and Rao~ \cite[Theorem~3.2]{AR} as well as Scheicher and Thuswaldner~ \cite[Theorem~5.8]{ST} proved that if $p$ is a monic polynomial with non-negative integer coefficients and such that $L(p)<2p(0)$ then $\big(p,\{0,1,\ldots, p(0)-1\}\big)$ is a GNS
over $\Zz$ with the finiteness property.
We show that Theorem~\ref{th:1} is a generalization of this assertion.

Indeed, with $\OO=\Z$ the set $\FF=[0,1)$ is a fundamental domain for $\OO\oz\Rr/\OO$. Moreover
$$
\DD_{\FF ,p(0)}=p(0)\cdot [0,1)\cap \Z = \{0,1,\ldots,p(0)-1\},
$$
which we denote by $\DD$. We have $\Delta_{\FF}=\{-1,0,1\}$, thus
\begin{equation}\label{AR1}
  \bigcup_{\delta \in \Delta_{\FF}} (\DD + f(0)\delta)=[-p(0),2p(0)-1]\cap \Z
\end{equation}
and
\begin{equation}\label{AR2}
\DD \cup (\DD - p(0)) =[-p(0),p(0)-1] \cap \Z.
\end{equation}

Let $f=p_nx^n +p_{n-1}x^{n-1}+\ldots+p_0$ with $p_n=1$. As all coefficients of $p$ are non-negative we have
$$
\max\{|w|\;:\; w\in Z_{\FF,p}\} = \sum_{j=1}^n p_j = L(p)-p(0) < p(0)
$$
by our assumption. This together with \eqref{AR2} implies immediately (ii). We have also the inequalities
$$
-p(0)+1 \le \min Z_{\FF,p}+ \DD < \max Z_{\FF,p}+ \DD \le 2p(0)-2.
$$
Comparing this with \eqref{AR1} we obtain (i). Finally let $J\subseteq\{1,\ldots, n\}$. Then $0\le \sum_{j\in J}p_j \le f(0)-1$, thus (iii) holds too and our claim is proved.
\end{rem}

\section{The finiteness property for large constant terms}\label{s4}

B. Kov\'acs~\cite[Section~3]{Kovacs} proved that if $p(x)\in \Z[x]$ is monic and irreducible, then $p(x+m)$ is a CNS polynomial for all sufficiently large integers $m$. Peth\H{o} \cite{pethoe1991:polynomial_transformation_and} pointed out that the irreducibility assumption is not essential. Peth\H{o} and Thuswaldner \cite{PT:2017} proved a generalization of Kov\'{a}cs' result
for GNS over orders in algebraic number fields. In this section,
we generalize their result further to GNS over arbitrary orders.

Let $\OO$ be an order of rank $d$.
We endow $\OO\oz\Rr$ with
a submultiplicative vector norm $\|\cdot\|$,
\ie $\|\alpha\cdot\beta\|\leq\|\alpha\|\cdot\|\beta\|$ for
$\alpha ,\beta\in\OO\oz\Rr$. For instance, if we choose a
$\Zz$-basis $\{ 1=\omega_1,\omega_2\kdots\omega_d\}$ of $\OO$,
then it is also an $\Rr$-basis of $\OO\oz\Rr$, and we may represent
$\alpha\in\OO\oz\Rr$ by the matrix $M_{\alpha}$
of the linear map $x\mapsto\alpha\cdot x$ with respect to this basis. Then we may choose as norm
$\|\alpha\|=\sum_{i,j} |a_{ij}|$
where the $a_{ij}$ are the entries of $M_{\alpha}$.
But in fact, in our arguments below, the choice of submultiplicative
norm doesn't matter.

For a subset $S\subset\OO\oz\Rr$ and $\varepsilon >0$, we define
the \emph{$\varepsilon$}-neighborhood of $S$ by
\[
(S)_\varepsilon := \{ \xi\in\OO\oz\Rr :\, \|\xi-\eta\| < \varepsilon \mbox{ for some } \eta\in S \}.
\]

Let $\FF$ be a bounded fundamental domain for $\OO\oz\Rr/\OO$ with
$0\in\FF$.
Let $p\in\OO[x]$ be a monic polynomial of degree $n$.
For $\alpha\in\OO$ we define
$p_{\alpha}(x):=p(x+\alpha)$.
Let $\GG_{p,\FF}$ be the family of GNS, consisting of those pairs
$(p_{\alpha},\DD_{\FF ,p(\alpha )})$ such that $p(\alpha )$ is not
a zero divisor of $\OO$.

With this notation we prove the following theorem
on $\GG_{p,\FF}$, which is a generalization of \cite[Theorem 4.1]{PT:2017}. As usual, we identify $r\in\Rr$ with $r\cdot 1$ where
$1$ is the unit element of $\OO$, and thus view $\Rr$ as a subfield
of $\OO\oz\Rr$.

\begin{theorem}\label{th:newKovacs}
Let $\OO$ be an order, $p\in\OO [x]$ a monic polynomial of degree $n$,
and $\FF$ a bounded fundamental domain for $\OO\oz\Rr/\OO$. Choose a submultiplicative norm $\|\cdot\|$ on $\OO\oz\Rr$.
Suppose that there is $\varepsilon >0$ such that
\begin{eqnarray}\label{condition-1}
&&\{ \xi\in \OO\oz\Rr :\, \|\xi\|<\varepsilon\}\subset \FF\cup (\FF -1);
\\
\label{condition-2}
&&\{ \xi\in\OO\oz\Rr:\,
\mbox{there is $r\in\Rr$ with $0<r<\varepsilon$,
$\|r^{-1}\xi -1\|<\varepsilon$}\}\subset\FF .
\end{eqnarray}
Then there is $\eta >0$ such that $(p_{\alpha},\DD_{\FF ,p(\alpha )})$
is a GNS with the finiteness property whenever
\begin{equation}\label{condition-3}
\alpha\in\OO,\ \ \ \|m^{-1}\alpha -1\|<\eta \ \mbox{for some rational integer } m>\eta^{-1} .
\end{equation}
\end{theorem}

\noindent
The somewhat more complicated condition \eqref{condition-2}
means that $\FF$ contains a cone emanating from $0$
around a small piece of the positive real line.

\begin{rem}\label{rem:42}
This implies that if $\FF$ is a fundamental domain for $\OO\oz\Rr$
satisfying \eqref{condition-1} and \eqref{condition-2},
then the family $\mathcal{G}_\mathcal{F}$ contains infinitely many GNS with the finiteness property.
\end{rem}

\begin{rem}\label{rem:43}
Condition \eqref{condition-3} cannot be weakened to say $\|\alpha\|$
being sufficiently large. For instance, let $\OO$ be an order
in a number field $K$ not equal to $\Qq$ or an imaginary quadratic field, and choose any bounded fundamental domain $\FF$ with $0\in\FF$
for $\OO\oz\Rr/\OO$.
For $\alpha\in \OO$ define
$\|\alpha\|$ to be the maximum of the absolute values of the conjugates of $\alpha$. This defines a submultiplicative norm on $\OO$,
which we extend to $\OO\oz\Rr$.
Let $p\in\Zz [x]$ be a monic polynomial with $p(0)=0$.
For $\alpha\in\OO$,
denote by $\alpha^{(1)}\kdots\alpha^{(d)}$
the conjugates of $\alpha$. Then $Np_{\alpha}=\prod_{i=1}^d p(x+\alpha^{(i)})$. If one of the conjugates of $\alpha$ has absolute value $\leq 1$ then $Np_{\alpha}$ is not expansive and
hence $(p_{\alpha}, \DD_{\FF ,p(\alpha )})$ cannnot have the
finiteness property. But $\OO$ has elements $\alpha$ with
one of the conjugates of absolute value $\leq 1$ and
$\|\alpha\|$ arbitrarily large.
\end{rem}

We assume that $\rank\OO =d$, $\deg p =n$.
By Taylor's formula,
\begin{eqnarray}\label{taylor}
p_{\alpha}(x)&=&x^n+p_{n-1}(\alpha)x^{n-1}+\cdots +p_0(\alpha )
\\
\nonumber
&&\quad\mbox{where } p_j(\alpha ):=\frac{p^{(j)}(\alpha )}{j!}\ \mbox{for } j=0\kdots n-1.
\end{eqnarray}
Note that $p_j(\alpha )\in\OO$ for $j=0\kdots n-1$.

By expressions $O(r)$, with $r$ a positive real, we denote elements $\xi$ of
$\OO\oz\Rr$ such that $\|\xi\|\leq C\cdot r$, where $C$ is a constant depending
only on $\OO$, $p$, $\FF$ and $\|\cdot\|$.

We start with a simple lemma.

\begin{lemma}\label{pol-estimates}
For $\eta$ sufficiently small in terms of $\OO ,\FF ,p,\|\cdot\|$
we have the following: if $\alpha\in\OO$ satisfies \eqref{condition-3}
for some positive integer $m$, then
\begin{eqnarray*}
&& p_j(\alpha )=\medbinom{n}{j}m^{n-j}(1+O(\eta ))\ \
\mbox{for } j=0\kdots n-1,
\\
&&\mbox{$p(\alpha )$ is not a zero divisor of $\OO$, \ie invertible in $\OO\oz\Rr$.}
\end{eqnarray*}
\end{lemma}

\begin{proof}
Write $\alpha =m+\beta$.
Then $\|\beta\| <\eta m$ and $m>\eta^{-1}$.
Assume $\eta <1$ as we may.
Then since $p$ is monic of degree $n$,
we have
\begin{eqnarray*}
p_0(\alpha )=p(\alpha )&=&p(m)+O(\beta m^{n-1}+\beta^2m^{n-2}+\cdots +\beta^n)
\\
&=&m^n+O(m^{n-1}+\beta m^{n-1}+\beta^2m^{n-2}+\cdots +\beta^n)
\\
&=& m^n+O(\eta m^n)=m^n(1+O(\eta )),
\end{eqnarray*}
using \eqref{condition-3}. Here the constants implied by the $O$-terms
depend on the norms of the coefficients of $p$.
The identities for $j>0$ are proved in the same manner.
To prove that $p(\alpha )$  has an inverse in $\OO\oz\Rr$,
write $\medfrac{p(\alpha )}{m^n}=1+\gamma$. Assuming $\eta$ is
sufficiently small we have $\|\gamma \|<1$.
The series $(1+\gamma )^{-1}=\sum_{k=0}^{\infty} (-1)^k\gamma^k$ converges
in $\OO\oz\Rr$,
since $\|\gamma^k\|\leq \|\gamma\|^k$ for $k\geq 0$. So
$\medfrac{p(\alpha )}{m^n}$ is invertible in $\OO\oz\Rr$.
Hence $p(\alpha )$
is invertible in $\OO\oz\Rr$.
\end{proof}

\begin{proof}[Proof of Theorem \ref{th:newKovacs}]

We proceed to show that if $\eta$ is sufficiently small, and $\alpha\in\OO$ satisfies
\eqref{condition-3}, then $(p_{\alpha},\DD_{\FF ,p(\alpha )})$ satisfies
conditions (i),(ii),(iii) of Theorem \ref{th:1}, with $p_{\alpha}$, $p_j(\alpha )$,
$\DD_{\FF ,p_{\alpha}}$
instead of $p$, $p_j$, $\DD$. This implies that $(p_{\alpha},\DD_{\FF ,p(\alpha )})$
is a GNS with the finiteness property.
We first observe that by the definition of $\FF$-neighbours,
\[
\overline{\FF}\cap\Big(\bigcup_{\delta\in\OO\setminus \Delta_{\FF}} (\FF+\delta )\Big)=\emptyset.
\]
From this fact and the boundedness of $\FF$ we infer
that after shrinking the $\varepsilon$ from conditions \eqref{condition-1},
\eqref{condition-2} if necessary,
\[
(\FF )_{\varepsilon}\cap \Big(\bigcup_{\delta\in\OO\setminus \Delta_{\FF}} (\FF+\delta )\Big)=\emptyset.
\]
The set $\OO\oz\Rr$ is a disjoint union of all the translates $\FF +\delta$
with $\delta\in\OO$ since $\FF$ is a fundamental domain for $\OO\oz\Rr/\OO$.
So in addition to \eqref{condition-1}, \eqref{condition-2} we have
\begin{equation}\label{condition-4}
(\FF )_{\varepsilon}\subseteq\bigcup_{\delta\in\Delta_{\FF}}(\FF+\delta ).
\end{equation}
We observe that if there is $\varepsilon>0$ for which \eqref{condition-1}, \eqref{condition-2} and \eqref{condition-4} hold, then it can be chosen to depend
only on $\OO$, $\FF$, $p$ and $\|\cdot\|$.

Let $\alpha\in\OO$ satisfy \eqref{condition-3} for some positive integer $m$,
where $\eta$ is a real that is sufficiently small in terms of $\OO$, $\FF$, $p$,
$\|\cdot\|$.
Note that condition (i) of Theorem \ref{th:1} (with $p ,p_j ,\DD$ replaced by $p_{\alpha}$,
$p_j(\alpha), \DD_{\FF ,p(\alpha )}=p(\alpha )\FF\cap\OO$) is equivalent to
\[
\Big(\sum_{j=1}^n \delta_jp_j(\alpha )\Big)+p(\alpha )\FF\cap\OO
\subset\bigcup_{\delta\in\Delta_{\FF}}\big(p(\alpha )\delta +p(\alpha )\FF\cap\OO\big)
\]
for all $\delta_j\in\Delta_{\FF}$ for $j=1\kdots n$. This follows, once we have
established that
\[
\Big(\sum_{j=1}^n \delta_jp_j(\alpha )\Big)+p(\alpha )\FF
\subset\bigcup_{\delta\in\Delta_{\FF}}\big(p(\alpha )\delta +p(\alpha )\FF\big)
\]
which in turn is equivalent to
\begin{equation}\label{condition-1a}
\Big(\sum_{j=1}^n \delta_j\cdot\frac{p_j(\alpha )}{p(\alpha )}\Big)+\FF
\subset
\bigcup_{\delta\in\Delta_{\FF}}\big(\FF +\delta\big)
\end{equation}
for all $\delta_j\in\Delta_{\FF}$, $j=1\kdots n$.
Here we have used that by Lemma \ref{pol-estimates}
$p(\alpha )$ has an inverse in
$\OO\oz\Rr$.
So in order to deduce (i) of Theorem \ref{th:1} it suffices to deduce \eqref{condition-1a}. But from Lemma \ref{pol-estimates}
and condition \eqref{condition-3} it follows that
\begin{equation}\label{norm-estimate}
\sum_{j=1}^n \delta_j\cdot\frac{p_j(\alpha )}{p(\alpha )}=O(m^{-1})=O(\eta )
\end{equation}
and for $\eta$ sufficiently small in terms of
$\OO$, $\FF$, $p$, $\|\cdot\|$ and $\varepsilon$, the left-hand side has norm
smaller than $\varepsilon$. So \eqref{condition-1a} follows from
\eqref{condition-4}. Hence condition (i) of Theorem \ref{th:1} is satisfied.

By a similar argument as above, it follows that condition (ii) of Theorem \ref{th:1}
follows, once we have
\begin{equation}\label{condition-2a}
\sum_{j=1}^n \delta_j\cdot\frac{p_j(\alpha )}{p(\alpha )}\in \FF\cup (\FF -1)
\end{equation}
for all $\delta_j\in\Delta_{\FF}$, $j=1\kdots n$. But
this clearly follows from \eqref{norm-estimate} and \eqref{condition-1} if $\eta$ is sufficiently
small in terms of $\OO$, $\FF$, $p$, $\|\cdot\|$ and $\varepsilon$.
This establishes condition (ii) of Theorem \ref{th:1}.

Finally, condition (iii) of Theorem \ref{th:1} follows once we have shown that
\begin{equation}\label{condition-3a}
\xi :=\sum_{j\in J}\frac{p_j(\alpha )}{p(\alpha )}\in\FF
\end{equation}
for every non-empty subset $J$ of $\{ 1\kdots n\}$. By Lemma \ref{pol-estimates},
\begin{equation}\label{estimation}
\xi =\big(\sum_{j\in J}\medbinom{n}{j}m^{-j}\big)(1+O(\eta )).
\end{equation}
Let $r:=\sum_{j\in J}\medbinom{n}{j}m^{-j}$.
Assuming $\eta$ is sufficiently small, and using $m^{-1}\leq\eta$
by \eqref{condition-3},
we get $0<r<\varepsilon$ and $\|r^{-1}\xi -1\|<\varepsilon$, which by
\eqref{condition-2} implies $\xi\in\FF$.
So condition (iii) of Theorem \ref{th:1} is also satisfied.

It follows that $(p_{\alpha} ,\DD_{\FF,p(\alpha)})$ has the finiteness property.
\end{proof}

Theorem~\ref{th:newKovacs} has the following variation.

\begin{theorem}\label{cor:newKovacs}
Let $\OO$ be an order, $p\in\OO [x]$ a monic polynomial of degree $n$,
and $\FF$ a bounded fundamental domain for $\OO\oz\Rr/\OO$. Choose a submultiplicative norm $\|\cdot\|$ on $\OO\oz\Rr$.
Suppose that $0$ is an interior point of $\FF$.
Then there is $\eta >0$ such that $(p_{\alpha},\DD_{\FF ,p(\alpha )})$
is a GNS with the finiteness property whenever
\begin{equation}\label{condition-3x}
\alpha\in\OO,\ \ \ \| m^{-1}\alpha -1\|<\eta\ \mbox{for some rational integer } m\ \mbox{with } |m|>\eta^{-1} .
\end{equation}
\end{theorem}

\begin{proof}
The proof is exactly the same as that of Theorem \ref{th:newKovacs},
except for the proof of \eqref{condition-3a}.
To prove this, let $\varepsilon >0$ be such that $\FF$ contains all the elements
of $\OO\oz\Rr$ of norm smaller than $\varepsilon$.
Pick $\alpha\in\OO$ satisfying \eqref{condition-3x}.
Then estimate \eqref{estimation} implies that if $\eta$ is sufficiently
small, then the $\xi$ from \eqref{condition-3a} has norm smaller than $\varepsilon$,
hence lies in $\FF$.
\end{proof}

\begin{rem} \label{rem:44}
Under the conditions of Theorem~\ref{th:newKovacs} there exists
a positive integer $N$ such that
$(p_m,\DD_{\FF ,p(m)})$
is a GNS with the finiteness property for $m\geq N$, while under the more restrictive conditions of Theorem~\ref{cor:newKovacs} there exists a positive integer $N$ such that
$(p_{\pm m},\DD_{\FF ,p(\pm m)})$
are GNS with the finiteness property for $m\geq N$.
\end{rem}

\begin{ex}\label{ex:Kovacs0}
Let $\OO =\Zz\times\Zz =\Zz^2$ with coordinatewise addition
and multiplication, zero element $(0,0)$ and unit element $(1,1)$.
Note that $\OO\oz\Rr =\Rr^2$ with coordinatewise addition and multiplication. Endow $\Rr^2$ with the maximum norm.
Take
\begin{eqnarray*}
\FF &=&\{ (x,y)\in\Rr^2:\, 0\leq x<1,\, -\half\leq y-x <\half\}
\\
&=&\{ x(1,1)+z(0,1):\, 0\leq x<1,\, -\half\leq z<\half\}.
\end{eqnarray*}
Then $\FF$ is a fundamental domain for $\Rr^2/\Zz^2$.
Let $p\in\OO [x]$ be a monic polynomial of degree $n$.
The coefficients of $p$ are pairs $(a,b)\in\Zz^2$, its leading
coefficient being the unit element $(1,1)$ of $\Zz^2$.
Thus we can write $p=(p_1,p_2)$, where $p_1,p_2$ are monic polynomials
in $\Zz [x]$ of degree $n$.
It is easy to see that if the integer $m$ is large enough then $p_1(m),p_2(m)>0$ and the corresponding digit set is
\[
\DD_{\FF ,p(m)}=\{ (x,y)\in\Zz^2:\, 0\leq x<p_1(m),\,
-\half\leq p_2(m)^{-1}y-p_1(m)^{-1}x<\half\}.
\]
One easily verifies that $\FF$ satisfies conditions
\eqref{condition-1}, \eqref{condition-2}, but note that
the inequality $\| r^{-1}\xi -1\|<\varepsilon$ in
\eqref{condition-2} has to be interpreted as $\| r^{-1}\xi-(1,1)\|<\varepsilon$, as $(1,1)$ is the unit element of $\OO$.
So by Theorem \ref{th:newKovacs}, $(p_m,\DD_{\FF ,p(m)})$ has the finiteness property for every sufficiently large $m$.
This means that for every pair of polynomials $a_1,a_2\in\Zz [x]$
there are $L>0$ and pairs $(d_i,d_i')\in\DD_{\FF ,p(m)}$
for $i=1\kdots L$, such that
\[
\mod{a_1}{\sum_{i=0}^L d_ix^i}{p_{1,m}},\ \
\mod{a_2}{\sum_{i=0}^L d_i'x_i}{p_{2,m}},
\]
where $p_{i,m}(x)=p_i(x+m)$ for $i=1,2$.
\end{ex}

\begin{ex}\label{ex:Kovacs1}
Let $\OO$, $p$ be as above but now take
$\FF =[0,1)\times [0,1)$.
This $\FF$ is a fundamental domain for $\Rr^2/\Zz^2$, but it does not satisfy condition \eqref{condition-1}, so Theorem \ref{th:newKovacs}
is not directly applicable. However, the corresponding digit set
can be expressed as a cartesian product
\[
\DD_{\FF ,p(m)}=\DD_{[0,1),p_1(m)}\times\DD_{[0,1),p_2(m)}=\{0,\ldots,p_1(m)-1\}\times \{0,\ldots,p_2(m)-1\}.
\]
The interval $[0,1)$ does satisfy \eqref{condition-1}, \eqref{condition-2}, with $\OO =\Zz$.
So by Theorem \ref{th:newKovacs}, for every sufficiently large positive
integer $m$, both GNS over $\Zz$, $(p_{i,m},\DD_{[0,1),p_i(m)})$ $(i=1,2)$
have the finiteness property. Now $(p_m,\DD_{\FF ,p(m)})$ is the
cartesian product of these two GNS (we assume the meaning of
this is obvious), and it easily
follows that it has the finiteness
property for every sufficiently large integer $m$.
\end{ex}

We will see in the next section that if we impose some other conditions
on $\FF$, then $(p_{-m},\DD_{\FF ,p(-m)})$ does not have the finiteness property for large $m$.

\begin{ex}\label{ex:Kovacs2}
Let $\OO$, $p$ be as above, but now take
$\FF= [-\half ,\half )\times[-\half ,\half )$; then $\FF$ is again a fundamental domain for $\Rr^2/\Zz^2$.
Let $m$ be a positive integer. Then
the corresponding digit set is
\begin{eqnarray*}
&&\DD_{\F,p(\pm m)}=\left\{-\left\lfloor \frac{|p_1(\pm m)|}{2}\right\rfloor, \dots,\left\lfloor \frac{|p_1(\pm m)|-1}{2}\right\rfloor\right\}\times
\\
&&\qquad\qquad\qquad\qquad\qquad\qquad
\times\left\{-\left\lfloor \frac{|p_2(\pm m)|}{2}\right\rfloor, \dots,\left\lfloor \frac{|p_2(\pm m)|-1}{2}\right\rfloor\right\},
\end{eqnarray*}
which is the product of two symmetric digit sets. The zero element
$(0,0)$ of $\Zz^2$ is obviously an interior point of $\FF$, hence, by Theorem \ref{cor:newKovacs} the pairs $(p_{\pm m},\DD_{\F,p(\pm m)})$ are GNS with the finiteness property for all large enough $m$.
\end{ex}


\section{GNS without the finiteness property}\label{s5}

We now prove a negative result on the finiteness property
for GNS over arbitrary orders $\OO$. We start with a simple
lemma that was proved in \cite[Lemma 5.1]{PT:2017}.
We state it here in a more general form.

\begin{lemma} \label{non-ECNS2}
Let $\OO$ be an order and $(p,\DD )$ a GNS over $\OO$.
If there exist a positive integer $h$, elements $d_0,d_1,\dots,d_{h-1}$ of $\DD$ not all $0$ and $q_1,q_2 \in \OO [x]$ with
    \begin{equation} \label{non-ECNSeq1}
   \sum_{j=0}^{h-1} d_j x^j  =  (x^{h} - 1) q_1(x) + q_2(x)p(x),
    \end{equation}
then $(p,\mathcal{D})$ does not have the finiteness property.
  \end{lemma}

\begin{proof}
Verbatim the same as the proof of Lemma 5.1 of \cite{PT:2017}.
\end{proof}

We can now prove the main result of the present section.
Recall that $p_{-m}(x):=p(x-m)$ and $\DD_{\FF ,p(-m)}:=p(-m)\FF\cap\OO$.
We view $\Rr$ as a subfield of $\OO\oz\Rr$ by identifying $r\in\Rr$
with $r\cdot 1$, where $1$ is the unit element of $\OO$.

\begin{theorem} \label{non-ECNSmain}
Let $\OO$ be an order, $\|\cdot\|$ a submultiplicative norm on
$\OO\oz\Rr$, and $\FF$ a bounded fundamental domain for $\OO\oz\Rr/\OO$
such that $0\in\FF$ and there is $\varepsilon >0$ such that
\begin{equation}\label{condition-5}
\{ \xi\in\OO\oz\Rr :\, \mbox{there is } r\in\Rr\ \mbox{with } 0<r<\varepsilon ,\
\| r^{-1}(1-\xi) -1\|<\varepsilon \}\subseteq\FF .
\end{equation}
Let $p\in\OO [x]$ be a monic polynomial. Then there exists a positive
integer $N$ such that for every $m>N$, $(p_{-m},\DD_{\FF ,p(-m)})$
does not have the finiteness property.
\end{theorem}

\noindent
Condition \eqref{condition-5} means that $\FF$ contains a cone
emanating from $1$, around the piece of the real line consisting
of all reals slightly smaller than $1$.

\begin{proof}
Let $d=\rank\OO$, $n=\deg p$.
We claim
that if $m$ is a large enough positive rational integer, then
$p_{-m}(0)=p(-m)\in D_{\mathcal{F},p(-m-1)}$.

Assume that our claim is true. Performing Euclidean division of
$p_{-m-1}$ by $x-1$ we obtain a polynomial $s_{m+1} \in \mathcal{O}[x]$ such that
\[
p_{-m-1} = (x-1)s_{m+1}+ p(-m).
\]
By the claim $p(-m)$ belongs to the digit set $D_{\mathcal{F},p(-m-1)}$ if $m$ is large enough. Applying Lemma~\ref{non-ECNS2} with $h=1, d_0=p(-m), q_1 = -s_{m+1}, q_2=1$, $p=p_{-m-1}$ and $\mathcal{D}=D_{\mathcal{F},p(-m-1)}$ we conclude that $(p_{-m-1} ,D_{\mathcal{F},p(-m-1)})$ is not a GNS with the finiteness property whenever $m$ is large enough.

Now we turn to prove the claim.
Write $p=x^n+p_{n-1}x^{n-1}+\cdots +p_0$. Note that
\begin{eqnarray}\label{first-p-relation}
p(-m)&=&(-m)^n+p_{n-1}(-m)^{n-1}+O(m^{n-2})
\\
\nonumber
&=&(-m)^n\big(1+p_{n-1}m^{-1}+O(m^{-2})\big).
\end{eqnarray}
This implies that for $m$ sufficiently large, we have
$\medfrac{p(-m)}{(-m)^n}=1+\gamma$ with $\|\gamma\|<1$.
The quantity $1+\gamma$ is invertible in $\OO\oz\Rr$
with inverse $1-\gamma +\gamma^2-\cdots$
so $p(-m)$ is invertible in $\OO\oz\Rr$ and
\begin{eqnarray}\label{second-p-relation}
p(-m)^{-1}&=&(-m)^{-n}(1-\gamma +\gamma^2-\cdots )
\\
\nonumber
&=&(-m)^{-n}\big( 1-p_{n-1}m^{-1}+O(m^{-2})\big).
\end{eqnarray}

Establishing our claim means proving that $p(-m)\in p(-m-1)\FF\cap\OO$ for every sufficiently large $m$. For this, it suffices to show that
\begin{equation}\label{quotient-relation}
\frac{p(-m)}{p(-m-1)}\in\FF
\end{equation}
for every sufficiently large $m$. By \eqref{first-p-relation}
and \eqref{second-p-relation} we have
\begin{eqnarray*}
\xi_m :=\frac{p(-m)}{p(-m-1)}&=&\frac{m^n}{(m+1)^n}\cdot(1+p_{n-1}m^{-1}+O(m^{-2}))\cdot
\\
&&\qquad\qquad\qquad\qquad\cdot
(1-p_{n-1}(m+1)^{-1}+O(m^{-2}))
\\
&=&\frac{m^n}{(m+1)^n}\big(1+O(m^{-2}))=1-n\cdot m^{-1}+O(m^{-2}).
\end{eqnarray*}
It is clear that for every sufficiently large $m$,
$\xi_m$ belongs to the set on the left-hand side of \eqref{condition-5},
hence $\xi_m\in\FF$.
This establishes
assertion \eqref{quotient-relation}, hence our claim and our
theorem.
\end{proof}

\begin{ex}
We continue the examination of the GNS given in Examples \ref{ex:Kovacs0}, \ref{ex:Kovacs1}.
As can be verified,
the fundamental domains $\FF$ from both examples satisfy \eqref{condition-5}
(the inequality $\|r^{-1}(1-\xi ) -1\|<\varepsilon$
being interpreted as $\|r^{-1}(1-\xi ) -(1,1)\|<\varepsilon$).
Hence in both cases, $(p_{-m},\DD_{\FF ,p(-m)})$ does not have the finiteness
property for every sufficiently
large $m$. Of course in the case of Example \ref{ex:Kovacs1},
we can alternatively appeal to a Cartesian product type argument.
\end{ex}

\section{Relation between power integral bases and GNS}\label{s6}

The theory of generalized number systems started with investigations in the ring of integers of algebraic number fields (see e.g.\ Brunotte, Huszti, and Peth\H{o}~\cite{BHP} for an overview and a list of relevant literature). This inspired \cite[Theorem~6.2]{PT:2017}, which states
if $\OO$ is a monogenic order of a number field,
then all but finitely many among the $\alpha$ with $\Zz [\alpha ]=\OO$ generate a number system with the finiteness property in $\mathcal{O}$. (All the exceptions are computable effectively.) This result forms an analogue of Kov\'acs and Peth\H{o} \cite[Theorem~5]{KovacsPetho} and is based on a result by Gy\H{o}ry~\cite{Gyory1976,Gyory} on the monogeneity of orders in number fields.
We generalize this to \'etale $\Q$-algebras.

A finite \'{e}tale $\Qq$-algebra $\Omega$ is up to isomorphism a direct product
of number fields $\K_1\times\cdots \times \K_{\ell}$, with coordinatewise
addition and multiplication. The \emph{degree} of $\Omega$
is the dimension of $\Omega$ as a $\Qq$-vector space,
notation $[\Omega :\Qq ]$.

We say that a finite \'etale $\Q$-algebra $\Omega$ is {\it effectively given}, if
it is given in the form $\K_1\times\dots \times \K_{\ell}$, where $\K_1,\dots,\K_{\ell}$ are effectively given as finite extensions of $\Q$,
{\it e.g.}, by minimal polynomials of primitive elements.
Further, an element $\alpha\in \Omega$ is said to be {\it effectively given} if in $\alpha=(\alpha_1,\ldots,\alpha_{\ell})$ with $\alpha_i\in \K_i$, the element $\alpha_i$ is effectively given as a $\Qq$-linear combination of powers of a given primitive
element of $\Kk_i$, see \cite[Sections 3.7 and 8.4]{Evertse_Gyory}.

An order of a finite \'{e}tale $\Qq$-algebra $\Omega$ is an order
$\OO\subset\Omega$ with $\rank\OO =[\Omega :\Qq ]$,
\ie $\Omega\cong\OO\oz\Qq$.
An \'{e}tale order is an order of any finite \'{e}tale
$\Qq$-algebra. By the Artin-Wedderburn Theorem, an order is \'{e}tale
if and only if it has no nilpotent elements.

We say that an \'{e}tale order $\OO$ is effectively given if the
\'{e}tale $\Qq$-algebra $\OO\oz\Qq$ is effectively given,
and if a $\Zz$-basis of $\OO$
is effectively given as a subset of $\OO\oz\Qq$.

We say that two elements $\alpha ,\beta$ of an order $\OO$ are $\Zz$-equivalent,
if $\beta =\alpha +u$ for some $u\in\Zz$. The order $\OO$ is called \emph{monogenic}
if $\OO =\Zz [\alpha ]$ for some $\alpha\in\OO$.
This is equivalent to $\OO$
having a \emph{power integral basis}, \ie a $\Zz$-basis of the form
$\{ 1,\alpha\kdots\alpha^{M-1}\}$ where $M=[\Omega :\Qq ]$.

We recall the following fundamental result.

\begin{proposition} \label{E-Gy}
Let $\OO$ be an effectively given \'{e}tale order.
Then it can be decided effectively whether $\OO$ is monogenic or not. Further, if $\OO$ is monogenic, then there exist only finitely many $\Z$-equivalence classes of $\alpha\in \OO$ such that $\OO=\Z[\alpha]$, and a complete set of representatives for these classes can be effectively determined.
\end{proposition}

Proposition \ref{E-Gy} does not hold for orders in general.
For instance, let $\OO =\Zz [x]/(x^3)$ and let $\alpha$ denote
the residue class of $x$. Then $\OO =\Zz [\alpha +b\alpha^2]$
for every $b\in\Zz$.

Proposition~\ref{E-Gy} is a special case of a more general effective result of Evertse and Gy\H{o}ry \cite[Corollary 8.4.7]{Evertse_Gyory} and allows to generalize the above mentioned \cite[Theorem 6.2]{PT:2017} from orders in number fields to \'{e}tale orders.

Let $\OO$ be an order. Recall that $\alpha\in\OO$ is not a zero divisor
of $\OO$ if and only if its norm $N(\alpha )$, that is the
determinant of the $\Zz$-linear map $x\mapsto \alpha x$
from $\OO$ to itself, is non-zero.
A  \emph{number system} for $\OO$ is a pair $(\alpha ,\DD )$,
where $\alpha\in\OO$ is not a zero divisor, and $\DD$
is a complete residue system of $\Zz$ modulo $N(\alpha )$.
We say that $(\alpha ,\DD )$ has the finiteness property
if every element of $\OO$ can be written as a finite sum
$\sum_{i=0}^L d_i\alpha^i$ with $d_i\in\DD$ for all $i$.
Note that this implies first that $\OO$ is monogenic,
and second, that $\Zz/N(\alpha )\Zz\cong\OO/\alpha\OO$,
so that $\DD$ is also a complete residue system of $\OO$
modulo $\alpha$.
Clearly,
if $p\in\Zz [x]$ is the minimal polynomial of $\alpha$, then
$(\alpha ,\DD )$ has the finiteness property if and only if
$(p,\DD )$ is a GNS over $\Zz$ with the finiteness property.

For $\alpha\in\OO$ such that $\OO =\Zz [\alpha ]$ and a fundamental domain $\FF$
for $\Rr/\Zz$ with $0\in\FF$, we define $\DD_{\FF ,\alpha }:=p(0)\FF\cap\Zz$,
where $p$ is the minimal polynomial of $\alpha$ over $\Qq$.

\begin{theorem} \label{thm:newKovacsPetho}
Let $\mathcal{O}$ be an \'{e}tale order.
Assume that $\OO$ is monogenic.
Let a bounded fundamental domain $\mathcal{F}$ for $\Rr/\Zz$
be given.
If $0$ is an inner point of $\FF$ then every $\alpha\in\OO$ with $\OO =\Zz [\alpha ]$,
with at most finitely many exceptions, gives rise to a number system
$(\alpha ,\DD_{\FF ,\alpha})$ with the finiteness property.
\end{theorem}

\begin{proof}
The proof is literally the same as the proof of \cite[Theorem 6.2]{PT:2017}.
\end{proof}

\begin{rem} \label{r:1}
In \cite{Evertse_Gyory}, a generalization of the above Proposition \ref{E-Gy} was proved dealing with the relative case as well, {\it i.e.}, with \'{e}tale orders of the shape
$\OO =\Zz_{\Kk}[\alpha ]$, where $\Zz_{\Kk}$
is the ring of integers of a number field $\Kk$.
To generalize Theorem~\ref{thm:newKovacsPetho} to this situation would require the generalization of Remark~\ref{rem:44} to all $m \in \mathcal{O}$ of which all
conjugates of $m$ are large enough. By Remark \ref{rem:43} such a generalization is only possible for special number fields.
\end{rem}

A finite set $\DD$ of integers can be a complete residue system modulo at most two integers, namely $\pm |\DD|$.  This does not hold any more for algebraic number fields with infinitely many units. Indeed, let $\K$ be such a field and $\Z_{\K}$ be its ring of integers. Let $\DD\subset \Z$ be given and assume that there exist $\alpha\in \Z_{\K}$ such that $\DD$ is a complete residue system modulo $\alpha$. Then $\DD$ is a complete residue system modulo $\alpha \varepsilon$ for each unit $\varepsilon\in \Z_{\K}$. From the next theorem it follows that there are only finitely many $\varepsilon\in \Z_{\K}$ such that the number system $(\alpha \varepsilon,\DD)$ has the finiteness property.

\begin{theorem} \label{th:givendigitset}
Let $\OO$ be an effectively given \'{e}tale order,
and $\DD$ a given finite subset of $\Zz$ containing $0$. Then there exist only finitely many, effectively computable $\alpha\in \OO$ such that the number system
$(\alpha,\DD)$ has the finiteness property.
\end{theorem}

\begin{proof}
Let $\alpha\in \OO$ and $\DD\subset \Z$ be such that the number system
$(\alpha,\DD)$ has the finiteness property. The set $\DD$ has to be a complete residue system of $\OO$ modulo $\alpha$, which is only possible if $|N(\alpha)| = |\DD|$.
If there is no such $\alpha$ then we are done. Otherwise, if $\OO\oz\Qq =\K_1\times\ldots \times\K_{\ell}$ and $\K_h$ are either the rational or an imaginary quadratic number field for all $h=1,\ldots,\ell$ then there are only finitely many $\alpha$ with $|N(\alpha)| = |\DD|$ and our assertion holds again.

We now assume that there are infinitely many $\alpha\in \OO$ such that $|N(\alpha)| = |\DD|$. If $(\alpha,\DD)$ has the finiteness property then there exist for all $\gamma \in \OO$ an integer $L$ and $d_i\in \DD, i=0,\ldots,L$ such that
  $$
  \gamma = \sum_{i=0}^{L} d_i \alpha^{i},
  $$
  hence $\OO$ is monogenic.
  By Proposition \ref{E-Gy} there exist only finitely many $\Z$-equiva\-lence classes of $\beta\in \OO$ such that $\OO=\Z[\beta]$. Hence there is such a $\beta$ and $u\in \Z$ with $\alpha=\beta+u$. For fixed $\beta$ there are only finitely many effectively computable $u\in \Z$ with $|N(\beta+u)|=|\DD|$, thus the assertion is proved.
\end{proof}

\bibliographystyle{siam}
\bibliography{orderCNS}

\end{document}